\documentclass[a4paper]{amsart}
\usepackage{amsmath,amstext,amsthm,amscd,amsopn,verbatim,amssymb, amsfonts}
\usepackage{hyperref}
\usepackage[bbgreekl]{mathbbol}
\usepackage{tikz}
\usetikzlibrary{matrix}
\usetikzlibrary{shapes}
\usetikzlibrary{arrows}
\usetikzlibrary{calc,3d}
\usetikzlibrary{decorations,decorations.pathmorphing}
\usetikzlibrary{through}
\tikzset{ext/.style={circle, draw,inner sep=1pt},int/.style={circle,draw,fill,inner sep=0, minimum size=5},nil/.style={inner sep=1pt}}
\tikzset{xst/.style={draw, cross out, minimum size=5, }}
\tikzset{exte/.style={circle, draw,inner sep=3pt},inte/.style={circle,draw,fill,inner sep=3pt}}
\tikzset{diagram/.style={matrix of math nodes, row sep=3em, column sep=2.5em, text height=1.5ex, text depth=0.25ex}}
\tikzset{diagram2/.style={matrix of math nodes, row sep=0.5em, column sep=0.5em, text height=1.5ex, text depth=0.25ex}}
\theoremstyle{plain}
  \newtheorem{thm}{Theorem}
  \newtheorem{defi}{Definition}
  \newtheorem{prop}{Proposition}
  
  \newtheorem{lemma}{Lemma}
\theoremstyle{definition}
  \newtheorem{ex}{Example}
  \newtheorem*{rem}{Remark}

\newcommand{\alg}[1]{\mathfrak{{#1}}}
\newcommand{\co}[2]{\left[{#1},{#2}\right]} 

\newcommand{\pd}[2]{ { \frac{\partial {#1} }{\partial {#2} } } }

\newcommand{\R}{{\mathbb{R}}}

\newcommand{\Gra}{{\mathsf{Gra}}}

\newcommand{\dGra}{{\mathsf{dGra}}}

\newcommand{\gl}{{\mathfrak{gl}}}
\newcommand{\id}{{\mathit{id}}}

\newcommand{\XGra}{{\mathsf{XGra}}}

\newcommand{\Lie}{\mathsf{Lie}}

\newcommand{\Def}{\mathrm{Def}}

\newcommand{\mO}{\mathcal{O}}
\newcommand{\GC}{\mathsf{GC}}
\newcommand{\fGC}{\mathsf{fGC}}
\newcommand{\fXGC}{\mathsf{fXGC}}
\newcommand{\dGC}{\mathsf{dGC}}
\newcommand{\dfGC}{\mathsf{dfGC}}
\newcommand{\XGC}{\mathsf{XGC}}

\newcommand{\vout}{\mathit{out}}

\newcommand{\bpm}{\begin{pmatrix}}
\newcommand{\epm}{\end{pmatrix}}
\newcommand{\Tpoly}{T_{\rm poly}}
\newcommand{\hatTpoly}{\hat T_{\rm poly}}
\newcommand{\tildeTpoly}{\tilde T_{\rm poly}}
\newcommand{\Dpoly}{D_{\rm poly}}

\newcommand{\grt}{{\mathfrak{grt}}}

\DeclareMathOperator{\End}{End}
\DeclareMathOperator{\Aut}{Aut}

\newcommand{\tsp}{\mathfrak{gsp}}

\newcommand{\MC}{\mathsf{Cone}} 

\begin{document}
\title{Stable cohomology of polyvector fields}
\author{Thomas Willwacher}
\address{Institute of Mathematics\\ University of Zurich\\ Winterthurerstrasse 190 \\ 8057 Zurich, Switzerland}
\email{thomas.willwacher@math.uzh.ch}


\begin{abstract}
We show that the stable cohomology of the algebraic polyvector fields on $\R^n$, with values in the adjoint representation is the symmetric product space on the cohomology of M. Kontsevich's graph complex, up to some known classes.
\end{abstract}
\maketitle

\section{Introduction}
Let $\Tpoly^{(n)}$ be the space of algebraic polyvector fields on $\R^n$. Concretely, $\Tpoly^{(n)}=\R[x^1,\dots ,x^n, \xi_1, \dots,\xi_n]$ is the space of polynomials in some degree 0 variables $x^1,\dots, x^n$ and degree 1 variables $\xi_1,\dots, \xi_n$. $\Tpoly^{(n)}$ is a Gerstenhaber algebra with the obvious commutative product and the bracket uniquely defined by the relations 
\[
 \co{\xi_i}{x^j}= \delta^j_i.
\]

There are inclusions of Gerstenhaber algebras
\[
 \cdots \to \Tpoly^{(n)} \to  \Tpoly^{(n+1)} \to \Tpoly^{(n+2)} \to \cdots.
\]
We define\footnote{In the following it will be important that the elements of $\Tpoly$ are polynomials and not power series, as are often considered in physics. The power series version requires different methods, cf. \cite{shoikhet, meoriented}.} 
\[
 \Tpoly = \lim_\rightarrow \Tpoly^{(n)} = \R[x^1,x^2,\dots , \xi_1, \xi_2, \dots].
\]
There is a natural subalgebra $\Tpoly^{\geq 2}\subset \Tpoly$ given by at least quadratic polynomials.
In the following, we will forget the graded commutative product on $\Tpoly$, we are only interested in the $\Lie\{1\}$ structure.\footnote{Here $\Lie\{1\}$ is the degree shifted Lie operad. Defining a $\Lie\{1\}$ structure on a space $V$ is equivalent to defining a $\Lie$ structure on the degree shifted space $V[1]$. The reader who likes $\Lie$ algebras better than $\Lie\{1\}$ algebras may replace all occurrences of $\Lie\{1\}$ by $\Lie$ and all occurrences of $\Tpoly$ by $\Tpoly[1]$. We work with $\Lie\{1\}$ here since it is more convenient regarding signs.} 
We define the graded symplectic $\Lie\{1\}$ algebra $\tsp(2n)\subset \Tpoly^{(n)}$ as the $\Lie\{1\}$ subalgebra spanned by the homogeneous quadratic polynomials.
M. Kontsevich has shown the following Theorem:
\begin{thm}[Kontsevich \cite{Kformal}]
\label{thm:konts}
 The $\Lie\{1\}$ algebra cohomology of $\Tpoly^{\geq 2}$ with values in the trivial representation is given by 
\[
 H(\Tpoly^{\geq 2}, \R) = \mathbf{S}( H_P(\tsp(2\infty), \R)\oplus H(\GC )[-2])
\]
Here $H_P(\tsp(2\infty),\R)$ is the primitive stable cohomology of the symplectic $\Lie\{1\}$ algebras (see below), $\mathbf{S}(\cdot)$ is the completed symmetric product space and $\GC$ is M. Kontsevich's graph complex, see \cite{megrt,Kgrcomplex} or section \ref{sec:grcomplex} below. 
\end{thm}
\begin{rem}
Concretely the primitive stable cohomology $H_P(\tsp(2\infty), \R)$ is
\[
H_P(\tsp(2\infty), \R) = \bigoplus_{n\equiv1\,\mathrm{ mod }\,4} \R w_n
\]
where $w_n$ is a cocyle of degree $n$.
The cocycles $w_n$ may also be identified with ``wheel'' graphs in a suitable graph complex, for example:
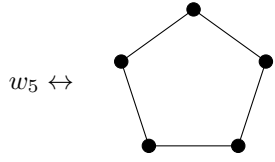
\[
\begin{tikzpicture}
\node at (-2,0) {$w_5\leftrightarrow$};
\draw (90:1) node[int] {} -- (162:1) node[int] {}
  -- (234:1) node[int] {} -- (306:1) node[int] {} -- (18:1) node [int] {} --cycle;
\end{tikzpicture}
\]
\end{rem}

Our goal here is to extend Theorem \ref{thm:konts} to Lie algebra cohomology of $\Tpoly$ with values in the adjoint representation. One remark is in order. There are three complexes which one may understand as the complex computing the stable cohomology of the $\Tpoly^{(n)}$, namely:
\begin{itemize}
 \item The Chevalley complex $C(\Tpoly, \Tpoly)$.
 \item The Chevalley complex $C(\Tpoly, \hatTpoly)$, where 
 \[
 \hatTpoly=\R[[x^1,x^2,\dots , \xi_1, \xi_2, \dots]]
 \]
 is the completed version of $\Tpoly$ with respect to the natural grading by degree.
  \item The Chevalley complex $C(\Tpoly, \tildeTpoly)$, where 
 \[
 \tildeTpoly=\lim_\leftarrow \Tpoly^{(n)}
 \]
 is the completed version of $\Tpoly$ with respect to the natural filtration by dimension.
\end{itemize}
The latter complex can be seen as a projective limit
\[
\lim_\leftarrow C(\Tpoly^{(n)}, \Tpoly^{(n)}).
\]
Here the maps 
\[
C(\Tpoly^{(n)}, \Tpoly^{(n)}) \leftarrow C(\Tpoly^{(n+1)}, \Tpoly^{(n+1)})
\]
are defined by using (i) the embedding $\Tpoly^{(n)} \to  \Tpoly^{(n+1)}$ from above and (ii) the map  $\Tpoly^{(n+1)} \to  \Tpoly^{(n)}$ of $\Tpoly^{(n)}$-modules obtained by setting $x_{n+1}=\xi_{n+1}=0$.
Note that we have the inclusions 
\[
C(\Tpoly, \Tpoly)\subset C(\Tpoly, \tildeTpoly) \subset C(\Tpoly, \hatTpoly).
\]

The result of this paper is that the cohomology of the second and third version is as follows.
\begin{thm}
\label{thm:main}
 The $\Lie\{1\}$ algebra cohomology of $\Tpoly$ with values in $\tildeTpoly$ or $\hatTpoly$ is
\[
 H(\Tpoly, \hatTpoly) =
  H(\Tpoly, \tildeTpoly) =
 \mathbf{S}( \R S[-2] \oplus \prod_{n\equiv1\,\mathrm{ mod }\,4}\R W_n[-2] \oplus H(\GC )[-2])[2].
\]
Here $W_n$ denotes the cocycle corresponding to a wheel graph with $n$ vertices (see below), $\mathbf{S}(\cdot)$ is the completed symmetric product space and $\GC$ is again M. Kontsevich's graph complex. The additional generator $S$ acts on a homogeneous polynomial $p\in \R[x^1,x^2,\dots , \xi_1, \xi_2, \dots]$ of $(x,\xi)$-degree $(r_1,r_2)$  as $S(p)=(r_1+r_2-2)p$.
\end{thm}

\begin{rem}
It has been shown by the author \cite{megrt} that $H^0(\GC )$ may be identified with the Grothendieck-Teichm\"uller Lie algebra $\grt_1$. It follows that the degree 0 cohomology in Theorem \ref{thm:main} above is spanned by $\grt_1$ and the additional generator $S$.
\end{rem}

The cohomology of $C(\Tpoly, \Tpoly)$ is more subtle, so we postpone its treatment to section \ref{sec:TpolyTpolycohom} below.
The definition of M. Kontsevich's graph complex will be recalled in section \ref{sec:grcomplex}. There it will also be recalled how graph cocycles can be mapped to $\Lie\{1\}$ algebra cocycles.


\begin{rem}
For simplicity we use $\R$ as our ground field. However, one may replace $\R$ by any field of characteristic zero.
\end{rem}

\subsection*{Acknowledgements}
This work benefitted from discussions with A. Khoroshkin and V. Dolgushev.
In particular I am grateful to V. Dolgushev for sharing an early version of his manuscript \cite{vasilystable} with me. Section \ref{sec:discussion} is dedicated to clarifying the relation of the present work with his. Most of this work was written while the author was a Junior Fellow of the Harvard Society of Fellows. The author is grateful for partial support by the Swiss National Science foundation, grant 200021\_150012 and the NCCR SwissMAP.

\section{Recollection: M. Kontsevich's graph complex}
\label{sec:grcomplex}
We recall here the definition of M. Kontsevich's graph complex \cite{Kgrcomplex}, following \cite{megrt} and \cite{megraphthings}.\footnote{The notation here will simplified compared to that of \cite{megrt}. In particular the graph complex $\GC$ here corresponds to $\GC_2$ there, and $\fGC$ here corresponds to $\fGC^\circlearrowleft$ there.}
Let us define the spaces
\[
 \Gra(n) := \prod_{k}\left( \{\text{$\R$-linear comb. of graphs with vertex set $[n]$ and edge set $[k]$} \}\otimes \R[k] \right)_{\mathbb{S}_k}.
\]
Here the notation is as follows. A \emph{graph} with vertex set $[n]:=\{1,\dots,n\}$ and edge set $[k]$ is an ordered set of unordered pairs of elements in $[n]$. 
We understand the $j$-th pair in this set as the $j$-th edge, or the edge labelled by $j$. Note that in particular, we allow \emph{tadpoles} or \emph{short cycles}, i.~e., edges of the form $(i,i)$.
 The symmetric group $S_k$ acts on such a graph by permuting the labels on edges. In the definition of $\Gra(n)$ we let $\mathbb{S}_k$ act also on $\R[k]$ by sign. Hence each edge contributes $-1$ to the degree, and exchanging the labels on two edges produces a minus sign.
For example, the following element of $\Gra(2)$ is zero by symmetry:
\[
\begin{tikzpicture}
\begin{scope}[xshift=-.3cm]
\node[ext] (e1) at (0,0) {1};
\node[ext] (e2) at (1,0) {2};
\draw (e1) edge[bend left] (e2)
           edge[bend right] (e2);
\end{scope}
\node at (1.35,0){= -};
\begin{scope}[xshift=2cm]
\node[ext] (e1) at (0,0) {1};
\node[ext] (e2) at (1,0) {2};
\draw (e1) edge[bend left] (e2)
           edge[bend right] (e2);
\end{scope}
\node at (3.8,0){=\, 0};
\end{tikzpicture}
\] 
In drawings we generally suppress the edge labels for simplicity, thus leaving a sign ambiguity.
The vector space $\Gra(n)$ carries a natural right action of $\mathbb{S}_n$ by permuting the vertices. In fact, the spaces $\Gra(n)$ assemble to form an operad $\Gra$. The operadic composition is given by ``inserting'' one graph at a vertex of another graph (see \cite[section 3]{megrt}, \cite{megraphthings}).
There is a natural action of $\Gra$ on $\Tpoly$. 
A graph $\Gamma\in \Gra(n)$ acts on polyvector fields $\gamma_1,\dots, \gamma_n\in \Tpoly$ by the formula
\[
 \Gamma(\gamma_1,\dots,\gamma_n) = \mu\circ \left( \prod_{(i,j)}\sum_{k=1}^d \pd{}{x^k_{(j)}} \pd{}{\xi_k^{(i)}} +\pd{}{x^k_{(i)}} \pd{}{\xi_k^{(j)}} \right)\left(\gamma_1\otimes\cdots\otimes\gamma_n \right).
\]
Here $\mu$ is the operation of multiplication of $n$ polyvector fields and the product runs over all edges $(i,j)$ in $\Gamma$, in the order given by the numbering of edges. The notation $\pd{}{x^k_{(j)}}$ means that the partial derivative is to be applied to the $j$-th factor of the tensor product, and similarly for $\pd{}{\xi_k^{(i)}}$.

There is a map of operads $\Lie\{1\}\to \Gra$, sending the generator to the graph 
\[
\begin{tikzpicture}
\node[ext] (e1) at (0,0) {1};
\node[ext] (e2) at (1,0) {2};
\draw (e1) edge (e2);
\end{tikzpicture}.
\]
The full graph complex is defined as the operadic deformation complex (see \cite[section 2]{megrt} for the definition) 
\[
\fGC := \Def((\Lie\{1\})_\infty\to \Gra)\cong \prod_{n=1}^\infty \Gra(n)^\mathbb{S_n}[2-2n].
\]
Elements are linear combinations of graphs, symmetric under permutations of the vertex numbers. In pictures, we indicate this by drawing black vertices without numbers:
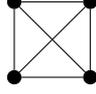
\[
\begin{tikzpicture}
\node[int] (e1) at (0,0) {};
\node[int] (e2) at (1,0) {};
\node[int] (e3) at (1,1) {};
\node[int] (e4) at (0,1) {};
\draw (e1) edge (e2) edge (e3) edge (e4)
      (e2) edge (e3) edge (e4)
      (e3) edge (e4);
\end{tikzpicture}
\] 
$\fGC$ is a differential graded Lie algebra, as is any deformation complex. The differential is the bracket with the Maurer-Cartan element 
\[
\begin{tikzpicture}
\node[int] (e1) at (0,0) {};
\node[int] (e2) at (1,0) {};
\draw (e1) edge (e2);
\end{tikzpicture}\, .
\]
This amounts to splitting vertices of graphs.
\begin{defi}
M. Kontsevich's graph complex $\GC$ is the differential graded Lie subalgebra of $\fGC$ spanned by connected graphs without tadpoles (short cycles) all of whose vertices are at least trivalent.
\end{defi}
It is shown in \cite{megrt} that this is indeed a differential graded Lie subalgebra.
Furthermore, it is shown in \cite[Proposition 3.4]{megrt} that the cohomology of $\fGC$ may be expressed through that of $\GC$ and the wheel classes $W_n$.
The graph cohomology is the cohomology of $\GC$ and is in general hard to compute. The main result of \cite{megrt} is the following Theorem:
\begin{thm}[\cite{megrt}]
$H^0(\GC )\cong \grt_1$ is the Grothendieck Teichm\"uller Lie algebra. Furthermore $H^{<0}(\GC)=0$.
\end{thm}
It is not known how to compute the higher graph cohomology. It is however known that it is not zero by computer experiments, see e.~g. \cite{barnatanmk}. 
Because of the action of $\Gra$ on $\Tpoly$ there is a natural map from $\fGC$ into the Chevalley complex $C(\Tpoly, \Tpoly)$ of $\Tpoly$:
\[
\fGC = \Def((\Lie\{1\})_\infty\to \Gra) \to
 \Def((\Lie\{1\})_\infty\to \End(\Tpoly)) \to C(\Tpoly, \Tpoly).
\]
This explains how graph cohomology classes can be understood as $\Lie\{1\}$ algebra cohomology classes (with values in the adjoint representation) in Theorem \ref{thm:main}.
It does not yet explain how graph cohomology classes can be understood as $\Lie\{1\}$ algebra cohomology classes of $\Tpoly^{\geq 2}$ with values in the trivial representation as in Theorem \ref{thm:konts}. However, note that by the inclusion $\Tpoly^{\geq 2}\to \Tpoly$ and by the map of $\Tpoly^{\geq 2}$-modules $\Tpoly\to \R$ by evaluating at 0 there is a natural map of complexes
\[
C(\Tpoly, \Tpoly) \to C(\Tpoly^{\geq 2}, \Tpoly) \to C(\Tpoly, \R).
\]
This means that graph cohomology classes also yield $\Lie\{1\}$ algebra cohomology classes of $\Tpoly^{\geq 2}$ with values in the trivial representation, as in Theorem \ref{thm:konts}.

\subsection{Directed graphs}
Instead of working with undirected graphs as above, one may work with directed graphs. The definitions from above carry over to this setting. In particular, we obtain an operad $\dGra$ of directed graphs, and a directed full graph complex
\[
\dfGC = \Def((\Lie\{1\})_\infty\to \dGra)\cong \prod_{n=1}^\infty \dGra(n)^\mathbb{S_n}[2-2n].
\]
It also acts on $\Tpoly$ in a natural way. There is a map of operads 
\[
\Gra\to \dGra
\]
sending an undirected graph to the sum of directed graphs obtained by assigning arbitrary directions to the edges.
Hence we also have a map of dg Lie algebras $\fGC\to \dfGC$.
We make the following definition 
\begin{defi}
The directed graph complex $\dGC$ is the sub-dg Lie algebra of $\dfGC$ spanned by connected graphs all of whose vertices are at least two-valent.
\end{defi}
Note that here we require vertices to be at least two-valent instead of trivalent. There is a map of dg Lie algebras $\GC\to \dGC$.
The following result was shown in \cite{megrt}.\footnote{A significantly weaker statement can also be found in \cite{ADR}. }
\begin{prop}
\label{prop:dGC}[see \cite[Appendix K]{megrt}\footnote{Note that that in loc. cit. different notation conventions regarding $\dGC$ are used.}]
 \[
  H(\dGC) \cong \bigoplus_{n\equiv1\,\mathrm{ mod }\,4} \R W_n \oplus H(\GC).
 \]
\end{prop}
Here $W_n$ corresponds to the wheel graph with $n$ vertices of cohomological degree $n-2$:
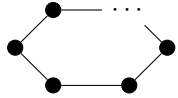
\[
\begin{tikzpicture}
[xst/.style={draw, cross out, minimum size=5, }, 
int/.style={draw, circle, fill, inner sep=2},
arr/.style={-triangle 60},]

\node [int] (v1) at (-1,0.5) {};
\node [int] (v2) at (-1.5,0) {};
\node [int] (v3) at (-1,-0.5) {};
\node [int] (v4) at (0,-0.5) {};
\node [int] (v5) at (0.5,0) {};
\draw (v1) edge (v2);
\draw (v2) edge (v3);
\draw (v3) edge (v4);
\draw (v5) edge (v4);
\node  (v6) at (0,0.5) {$\cdots$};
\draw (v1) edge (v6);
\draw (v5) edge (v6);
\end{tikzpicture}
\]

\begin{rem}
This means in particular that the cohomology of all the graph complexes above can be expressed using $H(\GC)$, up to some known classes.
\end{rem}

\section{xgraphs}
\label{sec:xgraphs}
\begin{defi}
 An \emph{xgraph} is a directed graph with vertex set $\{\vout\} \sqcup [n]$, where $[n]=\{1,\dots,n\}$ and $n$ can be $0,1,2,3,\dots$. The special vertex ``$\vout$'' we call the output vertex, the vertices $[n]$ we call the input vertices. The edges are labelled by numbers $1,2,\dots, k$, i.~e., the edge set is $[k]$.
\end{defi}
Note in particular that tadpoles are allowed. They are also allowed at the special vertex $\vout$.
A typical xgraph is shown in Figure \ref{fig:xgraph}. Note that we again suppress the edge labels to not clutter the picture too much.
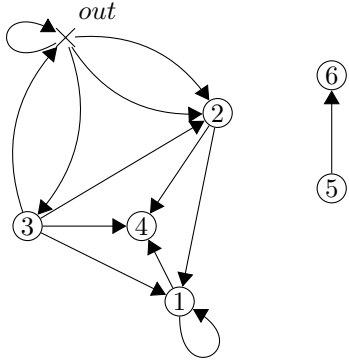
\begin{figure}
\centering
\begin{tikzpicture}
[
xst/.style={draw, cross out, minimum size=5, }, 
int/.style={draw, circle, fill, inner sep=2},
arr/.style={-triangle 60},
arrs/.style={-triangle 60},
]
\node [xst,label=70:{$out$}] (v1) at (0,0) {};
\node [ext] (v2) at (-0.5,-2.5) {3};
\node [ext] (v3) at (1.5,-3.5) {1};
\node [ext] (v4) at (1,-2.5) {4};
\node [ext] (v5) at (2,-1) {2};
\draw[arrs] (v1) edge[loop, in=150, out=210, looseness=0.8, loop , distance=1cm] (v1);
\draw[arrs] (v2) edge[bend left] (v1);
\draw[arrs] (v1) edge[bend left] (v2);
\draw[arr] (v2) edge (v3);
\draw[arr] (v3) edge (v4);
\draw[arr] (v5) edge (v4);
\draw[arr] (v2) edge (v5);
\draw[arr] (v2) edge (v4);
\draw[arr] (v5) edge (v3);
\draw[arrs] (v1) edge[bend left] (v5);
\draw[arrs] (v1) edge[bend right] (v5);
\draw[arr] (v3) edge[loop, in=-30, out=-90, looseness=0.8, loop , distance=1cm] (v3);
\node [ext] (v7) at (3.5,-0.5) {6};
\node [ext] (v6) at (3.5,-2) {5};
\draw[arr] (v6) edge (v7);
\end{tikzpicture}
\caption{\label{fig:xgraph} Some xgraph. It has three connected components after deleting $\vout$, namely: (i) the tadpole at $\vout$, (ii) the two vertices on the right and (iii) the remainder of the graph.}
\end{figure}
We distinguish three kinds of edges: (i) ``Normal edges'' between two input vertices. They will be considered to be of degree -1. (ii) ``Special edges'' between one input and the output vertex. They are assigned degree zero. (iii) ``Special tadpoles'', which are tadpoles at $\vout$. They will be assigned degree +1. 

We denote the number of those edges by $k_1, k_2, k_3$, i.e., $k=k_1+k_2+k_3$. There is a natural action of the symmetric groups $\mathbb{S}_n$ and $\mathbb{S}_k$ on any xgraph with $n$ input vertices and $k$ edges. We define the following vector spaces of xgraphs:
\begin{multline*}
 \XGra(n) := 
 \\
 \prod_{k}\left( \{\text{$\R$-linear comb. of xgraphs with $n$ input vert. and $k$ edges} \} [k_1-k_3] \right)_{\mathbb{S}_k}.
\end{multline*}
Here the action of $\mathbb{S}_k$ is by permutation of edge labels, with appropriate minus signs if odd edges are interchanged. In particular, xgraphs with $k_3\geq 2$ special tadpoles are zero by symmetry.

The spaces $\XGra(n)$ assemble to form an operad $\XGra$.
Let $\Gamma, \Gamma'$ be xgraphs then the composition  $\Gamma\circ_j\Gamma'$ is defined as follows:
\begin{enumerate}
\item If the number of incoming edges at vertex $j$ of $\Gamma$ is not equal to the number of outgoing edges at $\vout$ of $\Gamma'$ the result is zero. If the number of outgoing edges at vertex $j$ of $\Gamma$ is not equal to the number of incoming edges at $\vout$ of $\Gamma'$ the result is also zero. Otherwise proceed.
\item Delete the input vertex $j$ of $\Gamma$ and vertex $\vout$ of $\Gamma'$. this leaves several ``dangling'' edges in $\Gamma$ and $\Gamma'$.
\item The result is the sum over all possible xgraphs obtained by gluing outgoing dangling edges of $\Gamma$ to incoming dangling edges of $\Gamma'$ and vice versa. Here the numbering on edges is modified such that edges of $\Gamma$ have smaller labels than those of $\Gamma'$. 
\end{enumerate}

Here is an example of an operadic composition:
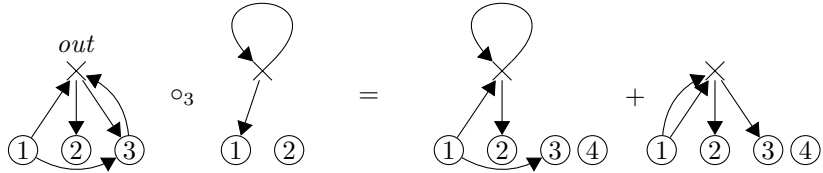
\[
\begin{tikzpicture}[every edge/.style={draw, -triangle 60}, scale=.7]
\tikzset{ext/.style={circle, draw,inner sep=1pt},int/.style={circle,draw,fill,inner sep=1pt},nil/.style={inner sep=1pt}}
\tikzset{xst/.style={draw, cross out, minimum size=5, }}

\node [ext] (v1) at (-4,0) {1};
\node [ext] (v3) at (-3,0) {2};
\node [ext] (v4) at (-2,0) {3};
\node [ext] (v6) at (0,0) {1};
\node [ext] at (1,0) {2};
\node [ext] (v7) at (4,0) {1};
\node [ext] (v9) at (5,0) {2};
\node [ext] (v10) at (6,0) {3};
\node [ext] at (6.7,0) {4};
\node [ext] (v11) at (8,0) {1};
\node [ext] (v13) at (9,0) {2};
\node [ext] (v14) at (10,0) {3};
\node [ext] at (10.7,0) {4};
\node [xst, label=90:{$\vout$}] (v2) at (-3,1.5) {};
\node [xst] (v5) at (0.5,1.5) {};
\node [xst] (v8) at (5,1.5) {};
\node [xst] (v12) at (9,1.5) {};
\draw  (v1) edge (v2);
\draw  (v2) edge (v3);
\draw  (v2) edge (v4);

\draw  (v4) edge[bend right] (v2);
\draw  (v5) edge[loop, looseness=15] (v5);
\draw  (v5) edge (v6);
\node at (-1,1) {$\circ_3$};
\node at (2.5,1) {$=$};
\node at (7.5,1) {+};
\draw  (v7) edge (v8);
\draw  (v8) edge (v9);
\draw  (v8) edge[loop, looseness=15] (v8);
\draw  (v1) edge[bend right] (v4);
\draw  (v7) edge[bend right] (v10);
\draw  (v11) edge (v12);
\draw  (v12) edge (v13);
\draw  (v12) edge (v14);
\draw  (v11) edge[bend left] (v12);
\end{tikzpicture}
\]

The operad $\dGra$ from above acts from the right on (the $\mathbb{S}$-module) $\XGra$ by ``insertions''. Furthermore there is a natural map of operads $\dGra\to \XGra$ by right action on the identity element $I\in \XGra(1)$. Concretely 
\begin{equation}
\label{equ:Idef}
 \begin{tikzpicture}
[xst/.style={draw, cross out, minimum size=5, }, 
int/.style={draw, circle, fill, inner sep=2},
arr/.style={-triangle 60},]

\node[xst] at (0,2) {};
\node [ext] at (0,1) {1};
\node [ext] (v1) at (1,1) {1};
\node [ext] (v4) at (2,1) {1};
\node [ext] (v6) at (3.2,1) {1};
\node [ext] (v8) at (4.2,1) {1};
\node [xst] (v2) at (1,2) {};
\node [xst] (v3) at (2,2) {};
\node [xst] (v5) at (3.2,2) {};
\node [xst] (v7) at (4.2,2) {};
\node at (0.5,1.5) {$+$};
\node at (-0.5,1.5) {$=$};
\node at (6.2,1.5) {``$=$''$\, \exp$};
\node at (-1,1.5) {$I$};
\node at (3.7,1.5) {$+$};
\node at (2.5,1.5) {$+ \, \frac{1}{2!}$};
\node at (1.5,1.5) {$+$};
\draw[arr] (v1) edge (v2);
\draw[arr] (v3) edge (v4);
\draw[arr] (v5) edge[bend right] (v6);
\draw[arr] (v5) edge[bend left] (v6);
\draw[arr] (v7) edge[bend left] (v8);
\draw[arr] (v8) edge[bend left] (v7);
\node at (5,1.5) {$+ \cdots$};
\begin{scope}[shift={(6.6,0)}]
\node [ext] (v1) at (1,1) {1};
\node [ext] (v4) at (2,1) {1};
\node [xst] (v2) at (1,2) {};
\node [xst] (v3) at (2,2) {};

\draw[arr] (v1) edge (v2);
\draw[arr] (v3) edge (v4);

\node at (1.5,1.5) {$+$};
\end{scope}
\node[scale=1.5] at (7,1.5) {\huge $($};
\node[scale=1.5] at (9.1,1.5) {\huge $)$};
\end{tikzpicture}.
\end{equation}
The sum is over all possible ways to connect $\vout$ to one vertex.
The coefficient of each term is $\frac{1}{k!l!}$ where $k$ and $l$ are the numbers of in- and outgoing edges at $\vout$.
From the map $\Lie\{1\}\to \dGra$ we hence get a map $\Lie\{1\}\to \XGra$. 
\begin{defi}
The \emph{full extended graph complex} is the deformation complex
\[
 \fXGC = \Def(\Lie\{1\}\to \XGra)\cong \prod_{n=1}^\infty \XGra(n)^\mathbb{S_n}[2-2n].
\]
The \emph{extended graph complex} is the subcomplex $\XGC\subset \fXGC$ spanned by graphs which are connected and non-empty after deleting the vertex $\vout$.\footnote{The non-emptiness condition excludes the graph consisting of just the vertex $\vout$.}
\end{defi}

$\fXGC$ is a dg Lie algebra (as is any deformation complex). $\XGC$ is a Lie subalgebra. Note that from the inclusions of operads $\Gra\to \dGra\to\XGra$ we obtain inclusions of dg Lie algebras $\GC\to\dGC\to \XGC$. 
Elements of $\fXGC(n)$ are series in xgraphs (modulo permutations of the edge labels), symmetric under interchange of the vertex labels. We indicate this in pictures by not drawing vertex labels but black dots instead, understanding that one should sum over all possible ways of putting vertex labels. For example
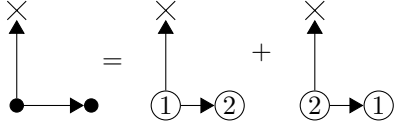
\[
\begin{tikzpicture}[every edge/.style={draw, -triangle 60}, scale=.7]

\node [xst] (v2) at (-0.4,0.6) {};
\node [xst] (v5) at (2.4,0.6) {};
\node [xst] (v8) at (5.2,0.6) {};
\node [int] (v1) at (-0.4,-1.2) {};
\node [int] (v3) at (1,-1.2) {};
\node [ext] (v4) at (2.4,-1.2) {1};
\node [ext] (v6) at (3.6,-1.2) {2};
\node [ext] (v7) at (5.2,-1.2) {2};
\node [ext] (v9) at (6.4,-1.2) {1};
\draw  (v1) edge (v2);
\draw  (v1) edge (v3);
\draw  (v4) edge (v5);
\draw  (v4) edge (v6);
\draw  (v7) edge (v8);
\draw  (v7) edge (v9);
\node at (1.4,-0.4) {=};
\node at (4.2,-0.2) {+};
\end{tikzpicture}
\] 

Let us discuss the differential on $\XGC$ (and $\fXGC$). It is defined as the bracket with the element 
\[
\begin{tikzpicture}
\node at (-.9,0) {$m=\,I\circ$};
\node[int] (e1) at (0,0) {};
\node[int] (e2) at (.6,0) {};
\draw (e1) edge (e2);
\end{tikzpicture}.
\]
The bracket produces two terms, which one can depict as follows.
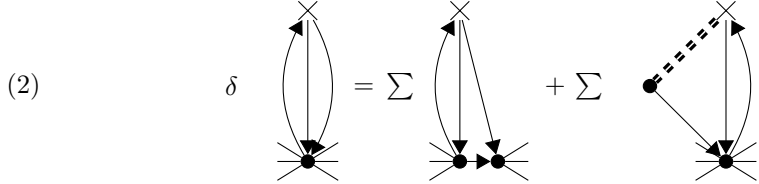
\begin{equation}
\label{equ:diffpict}
\begin{tikzpicture}[baseline=-.65ex,every edge/.style={draw, -triangle 60}]
\node [xst] (v1) at (-2.5,1) {};
\node [xst] (v3) at (-0.5,1) {};
\node [xst] (v6) at (3,1) {};
\node [int] (v2) at (-2.5,-1) {};
\node [int] (v4) at (-0.5,-1) {};
\node [int] (v5) at (0,-1) {};
\node [int] (v7) at (3,-1) {};
\node [int] (v8) at (2,0) {};
\draw  (v1) edge[bend left] (v2);
\draw  (v2) edge[bend left] (v1);
\node at (-3.5,0) {$\delta$};
\node at (-1.5,0) {$=\, \sum$};
\node at (1,0) {$+ \, \sum$};
\draw  (v1) edge (v2);
\draw  (v3) edge (v4);
\draw  (v3) edge (v5);
\draw  (v4) edge[bend left] (v3);
\draw  (v4) edge (v5);
\draw  (v6) edge (v7);
\draw  (v8) edge (v7);
\draw  (v7) edge[bend right] (v6);
\draw  (v6) edge[very thick, double, dashed, -] (v8);
\foreach \y in {.25,0,-.25}
{
\draw (v2) -- +(-.4, \y);
\draw (v2) -- +(.4, \y);
\draw (v4) -- +(-.4, \y);
\draw (v5) -- +(.4, \y);
\draw (v7) -- +(-.4, \y);
\draw (v7) -- +(.4, \y);
}
\end{tikzpicture}
\end{equation}
Here only one vertex and the output vertex are drawn, it should be understood that this is only a part of the bigger graph, and the differential acts on the other parts similarly.
To produce the first term on the right the vertex is split into two and the incoming edges are reconnected in an arbitrary way. To produce the second term one detaches one edge from $\vout$ and connects it to a newly created vertex. One furthermore sums over all ways of connecting this new vertex to $\vout$. In other words, the double dashed line stands for
\[
 \begin{tikzpicture}
[
arr/.style={-triangle 60},]

\node[xst] (x0) at (-1,2) {};
\node [int] (x1) at (-1,1) {};
\draw (x0) edge[very thick, double, dashed, -] (x1);
\node[xst] at (0,2) {};
\node [int] at (0,1) {};
\node [int] (v1) at (1,1) {};
\node [int] (v4) at (2,1) {};
\node [int] (v6) at (3.4,1) {};
\node [int] (v8) at (4.4,1) {};
\node [xst] (v2) at (1,2) {};
\node [xst] (v3) at (2,2) {};
\node [xst] (v5) at (3.4,2) {};
\node [xst] (v7) at (4.4,2) {};
\node at (0.5,1.5) {$+$};
\node at (-0.5,1.5) {$=$};
\node at (3.9,1.5) {$+$};
\node at (2.7,1.5) {$+\, \frac{1}{2!}$};
\node at (1.5,1.5) {$+$};
\node at (5.2,1.5) {$+ \cdots$};
\draw[arr] (v1) edge (v2);
\draw[arr] (v3) edge (v4);
\draw[arr] (v5) edge[bend right] (v6);
\draw[arr] (v5) edge[bend left] (v6);
\draw[arr] (v7) edge[bend left] (v8);
\draw[arr] (v8) edge[bend left] (v7);
\node at (6.2,1.5) {``$=$''$\, \exp$};
\begin{scope}[shift={(6.6,0)}]
\node [int] (v1) at (1,1) {};
\node [int] (v4) at (2,1) {};
\node [xst] (v2) at (1,2) {};
\node [xst] (v3) at (2,2) {};
\draw[arr] (v1) edge (v2);
\draw[arr] (v3) edge (v4);
\node at (1.5,1.5) {$+$};
\end{scope}
\node[scale=1.5] at (7,1.5) {\huge $($};
\node[scale=1.5] at (9.1,1.5) {\huge $)$};
\end{tikzpicture}.
\]
Note that the first term in particular produces a new vertex of valence one. This term will be important later.

\subsection{An alternative definition of \texorpdfstring{$\XGra$ and $\fXGC$}{XGra and fXGC}}
\label{sec:altdef}
Consider the endomorphism operad of $\Tpoly^{(n)}$, i.~e., $\End(\Tpoly^{(n)})$. It carries a natural action of $GL(n)$. Let us consider the space of $GL(n)$ invariants $\End(\Tpoly^{(n)})^{GL(n)}$. 
$\End(\Tpoly^{(n)})$ is made of products and symmetric products of the fundamental representation $\R^n$ and its dual $(\R^n)^*$. Hence, by classical invariant theory elements of $\End(\Tpoly^{(n)})^{GL(n)}$ can be written as certain graphs, where each edge stands for one pairing of $\R^n$ and $(\R^n)^*$. If one looks at the structure of the $GL(n)$ representation $\End(\Tpoly^{(n)})$, one sees that these graphs are xgraphs as defined above (apart from the labelling on edges). 
The operadic composition on $\XGra$ has been defined so that it agrees with the operadic composition on $\End(\Tpoly^{(n)})^{GL(n)}$.

Note that for finite $n$ not all xgraphs yield nonzero elements of $\End(\Tpoly^{(n)})^{GL(n)}$. Instead 
\[
\End(\Tpoly^{(n)})^{GL(n)} \cong \XGra/I_n
\]
where $I_n$ is an operadic ideal spanned by (linear combinations of) xgraphs acting as zero. Note that $\lim_\to I_n=0$. Hence we arrive at the following alternative definition of $\XGra$.

\vspace{2mm}
{\bf Alternative Definition:}
\[
\XGra = \lim_\leftarrow  \End(\Tpoly^{(n)})^{GL(n)}.
\]
\vspace{2mm}

From this alternative definition it is also clear that there is a map of operads $\dGra\to \XGra$, because $\dGra$ acts on $\Tpoly$ in a $GL(\infty)$ invariant way.

\subsection{The action on polyvector fields}
The action of the operad $\dGra$ on $\Tpoly^{(n)}$ factors through $\dGra\to \XGra$. The action of $\XGra$ on $\Tpoly^{(n)}$ is clear from the alternative definition in the previous subsection. Let us nevertheless describe it combinatorially for completeness. The action of an xgraph $\Gamma \in \XGra$ with $N$ input vertices on polyvector fields $\gamma_1,\dots, \gamma_N\in \Tpoly^{(n)}$ is given as follows. 
\begin{enumerate}
\item If $\Gamma$ has a special tadpole, then 
\[
\Gamma(\gamma_1,\dots, \gamma_N) = \pm\left(\sum_{j=1}^n x^j\xi_j\right)\wedge \Gamma'(\gamma_1,\dots, \gamma_N)
\]
where $\Gamma'$ is obtained from $\Gamma$ by deleting the special tadpole. The sign is ``+'' if the special tadpole is the first in the ordering of edges. So the special tadpole indicates multiplication with the Euler vector field.
\item Suppose that $\Gamma$ has no special tadpole. Suppose also that the $\gamma_1,\dots, \gamma_N$ are of homogeneous degree in the $x$ variables  and $\xi$ variables separately. Say $\gamma_j$ has $x$-degree $k_j$ and $\xi$-degree $l_j$. Let $\tilde \Gamma$ be the graph in $\dGra$ obtained by deleting the vertex $\vout$ and all adjacent edges. Then 
\[
\Gamma(\gamma_1,\dots, \gamma_N) 
=
\begin{cases}
\left(\prod_{j=1}^N k_j!l_j! \right)  \tilde \Gamma(\gamma_1,\dots, \gamma_N)
& \quad \text{if condition $(*)$ below is satisfied} \\
0 & \quad \text{otherwise}
\end{cases}
\]
Here condition $(*)$ is that the number of outgoing edges at vertex $j$ is $l_j$ and the number of incoming edges is $k_j$ for all $j=1,2,\dots, N$.
\end{enumerate}

By this action we automatically obtain a map of dg Lie algebras
\[
\XGC \to C(\Tpoly^{(n)},\Tpoly^{(n)}).
\]

Note that $\XGra$ in our definition does not act on the stable version $\Tpoly$ because the Euler vector field is not contained in it. However, the suboperad $\XGra^{nt}\subset \XGra$ of xgraphs without special tadpoles acts 
on $\Tpoly$. In particular, we get a map of dg Lie algebras 
\[
\XGC^{nt} \to C(\Tpoly,\Tpoly).
\]
where $\XGC^{nt}\subset \XGC$ is spanned by graphs without special tadpoles.

\section{The cohomology of \texorpdfstring{$\fXGC$ and $\XGC$}{fXGC and XGC}}
For an xgraph $\Gamma$ we define its set of connected components as the set of connected components of the graph obtained by deleting the vertex $\vout$. For example, the following graph has three connected components in this sense.
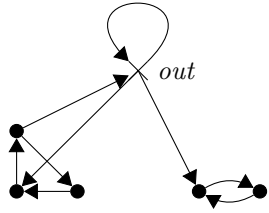
\[
\begin{tikzpicture}[every edge/.style={draw, -triangle 60}, scale=.8]
\node [xst, label=0:{$\vout$}] (v2) at (0,0) {};
\node [int] (v1) at (-2,-1) {};
\node [int] (v4) at (-2,-2) {};
\node [int] (v3) at (-1,-2) {};
\node [int] (v5) at (1,-2) {};
\node [int] (v6) at (2,-2) {};
\draw  (v1) edge (v2);
\draw  (v2) edge[loop, looseness=15] (v2);
\draw  (v3) edge (v4);
\draw  (v1) edge (v3);
\draw  (v4) edge (v1);
\draw  (v2.base) edge (v4);
\draw  (v5) edge[bend left] (v6);
\draw  (v6) edge[bend left] (v5);
\draw  (v2.base) edge (v5);
\end{tikzpicture}
\]
In particular, a special tadpole counts as one connected component. Note that the differential on $\fXGC$ leaves invariant the connected components.Hence the following result is immediate.
\begin{lemma}
\label{lem:reduction}
 \[
  H(\fXGC) \cong \mathbf{S}( H(\XGC )[-2])[2].
 \]
\end{lemma}

So let us focus on $\XGC$.  
The goal of this section is to show the following result.
\begin{prop}
\label{prop:XGC}
\[ 
H(\XGC)\cong  H(\dGC) \oplus \R S.  
\]
\end{prop}

Here the inclusion $\dGC\to \XGC$ appeared already in section \ref{sec:xgraphs}.
The last element $S$ is defined as
\[
S = U - 2I
\]
where $I\in \XGra(1)$ is the unit element (see \eqref{equ:Idef}) and $U$ is the following series of graphs:

\[
\begin{tikzpicture}
[
arr/.style={-triangle 60},scale=.9]

\node [int] (v1) at (-0.5,1) {};
\node [int] (v4) at (0.3,1) {};
\node [int] (v6) at (1.6,1) {};
\node [int] (v8) at (5.4,1) {};
\node [xst] (v2) at (-0.5,2) {};
\node [xst] (v3) at (0.3,2) {};
\node [xst] (v5) at (1.6,2) {};
\node [xst] (v7) at (5.4,2) {};
\node at (4.7,1.5) {$+\frac{3}{3!}$};
\node at (-1,1.5) {$=$};
\node at (-1.5,1.5) {$U$};
\node at (-0.1,1.5) {$+$};
\node at (2.1,1.5) {$+\, 2$};

\node at (0.9,1.5) {$+\, \frac{2}{2!}$};
\node at (3.3,1.5) {$+\, \frac{2}{2!}$};
\node at (6.2,1.5) {$+ \cdots$};
\draw[arr] (v1) edge (v2);
\draw[arr] (v3) edge (v4);
\draw[arr] (v5) edge[bend right] (v6);
\draw[arr] (v5) edge[bend left] (v6);
\draw[arr] (v7) edge[bend left] (v8);
\draw[arr] (v7) edge[bend right] (v8);

\node [xst] (v5) at (2.7,2) {};
\node [int] (v6) at (2.7,1) {};

\draw[arr] (v5) edge[bend right] (v6);
\draw[arr] (v6) edge[bend right] (v5);

\node [xst] (v5) at (4,2) {};
\node [int] (v6) at (4,1) {};

\draw[arr] (v6) edge[bend left] (v5);
\draw[arr] (v6) edge[bend right] (v5);

\draw[arr] (v7) edge (v8);

\node at (8.1,1.5) {``$=$''$\, \left.\frac{d}{d\lambda}\right|_{\lambda=1} \exp\, \lambda$};
\begin{scope}[shift={(9.4,0)}]
\node [int] (v1) at (1,1) {};
\node [int] (v4) at (2,1) {};
\node [xst] (v2) at (1,2) {};
\node [xst] (v3) at (2,2) {};
\draw[arr] (v1) edge (v2);
\draw[arr] (v3) edge (v4);
\node at (1.5,1.5) {$+$};
\node[scale=1.5] at (.4,1.5) {\huge $($};
\node[scale=1.5] at (2.5,1.5) {\huge $)$};
\end{scope}

\end{tikzpicture}
\]
in $\XGra(1)$. The operation $U$ corresponds to the scaling operator, acting on multivector fields of degree (joint in $x$'s and $\xi$'s) $k$ as multiplication by $k$.
Let us prove the proposition above. In fact, we will prove a somewhat stronger statement.

\vspace{2mm}
{\bf Claim 1:} The inclusion $\dGC \oplus \R S \to \XGC$ is a quasi-isomorphism.
\vspace{2mm}

To prove the claim, consider the descending complete filtration on $\XGC$ by the valence of $\vout$. 
\[
 \XGC= \mathcal F^0 \supset \mathcal F^1  \supset \mathcal F^2 \supset \cdots 
\]
where $\mathcal F^p$ is spanned by graphs for which $\vout$ has valence $\geq p$. The filtration is not compatible with the differential in the usual sense, but $\delta \mathcal F^p \subset \mathcal F^{p-1}$. Consider the associated   
graded. The leading part of the differential is 
\[
\delta_1: \mathcal F^p/\mathcal F^{p+1} \to \mathcal F^{p-1}/\mathcal F^{p}.
\]
This differential reduces the valence of $\vout$ by one. Considering 
``equation'' \eqref{equ:diffpict}, we see that the only part of $\delta$ that can achieve this is the following:
\[
\begin{tikzpicture}[every edge/.style={draw, -triangle 60}]
\node [xst] (v1) at (-2.5,1) {};
\node [xst] (v6) at (0.5,1) {};
\node [int] (v2) at (-2.5,-1) {};
\node [int] (v7) at (0.5,-1) {};
\node [int] (v8) at (-0.5,0) {};
\draw  (v1) edge[bend left] (v2);
\draw  (v2) edge[bend left] (v1);
\node at (-3.5,0) {$\delta_1$};
\node at (-1.5,0) {$=\, \sum$};
\draw  (v1) edge (v2);
\draw  (v6) edge (v7);
\draw  (v8) edge (v7);
\draw  (v7) edge[bend right] (v6);
\foreach \y in {.25,0,-.25}
{
\draw (v2) -- +(-.4, \y);
\draw (v2) -- +(.4, \y);
\draw (v7) -- +(-.4, \y);
\draw (v7) -- +(.4, \y);
}
\end{tikzpicture}
\]
So $\delta_1$ detaches one special edge from $\vout$ and connects it to a new vertex.
 Note that the same happens for special tadpoles:
\[
\begin{tikzpicture}[every edge/.style={draw, -triangle 60}]
\node [xst] (v1) at (-2.5,1) {};
\node [xst] (v6) at (-0.5,1) {};
\node [int] (v8) at (-0.5,2) {};
\node at (-3.5,1) {$\delta_1$};
\node at (-1.5,1) {$=$};
\begin{scope}[shift={(0.5,0)}]
\node [xst] (v9) at (0.5,1) {};
\node [int] (v10) at (0.5,2) {};
\end{scope}

\foreach \y in {.25,0,-.25}
{
\draw (v1.base) -- +(\y, -.5);
\draw (v6.base) -- +(\y, -.5);
\draw (v9.base) -- +(\y, -.5);
}
\draw  (v1) edge[loop, looseness=15] (v1);
\draw  (v8) edge (v6);

\draw  (v9.base) edge (v10);
\node at (0.2,1) {+};
\end{tikzpicture}
\]
Concretely, $\delta_1$ has the form (on a graph $\Gamma$)
\[
 \delta_1 \Gamma = \sum_{e} \Gamma_{e}.
\]
Here the sum runs over all special edges $e$. Say $e$ connects some vertex $v\neq \vout$ to $\vout$ ($\vout$ to $v$). Then $\Gamma_{e}$ is obtained by (i) adding a new vertex $v'$ to $\Gamma$ and (ii) reconnect the edge $e$ so as to connect $v$ to $v'$ ($v'$ to $v$). Note that the reconnected edge has degree -1, while previously it had degree 0. The edge $e$ becomes the first in the list of edges (all other edge labels are shifted by one).
For $v=\vout$, i.~e., for special tadpoles, the contribution consists of two terms as depicted above.

\vspace{2mm}
{\bf Claim 2:} The cohomology of the associated graded with differential $\delta_1$ is 
\[
H(\mathit{gr}\XGC, \delta_1)= \mathit{gr}^0\XGC / \delta_1 \mathit{gr}^1\XGC
=
\mathcal F^0 / (\mathcal F^1 + \delta_1 \mathcal F^1 )
\]
\vspace{2mm}

To see this one uses a combinatorial argument along the lines of the proof of Proposition 3 in \cite{megrt}, see also \cite{Kformal}. Let the \emph{core} of an xgraph be the isomorphism class of xgraphs obtained by identifying all valence 1 vertices with $\vout$.\footnote{An isomorphism of an xgraph is a permutation of the edge and vertex labels that maps the edge set to itself.}
Here is an example of an xgraph (left, labels not shown) and its core (right).
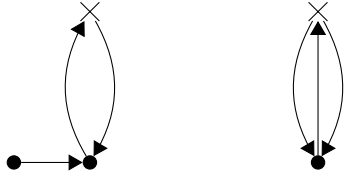
\[
\begin{tikzpicture}[every edge/.style={draw, -triangle 60}]
\node [xst] (v1) at (-2,1) {};
\node [xst] (v5) at (1,1) {};
\node [int] (v2) at (-2,-1) {};
\node [int] (v3) at (-3,-1) {};
\node [int] (v6) at (1,-1) {};
\draw  (v1) edge[bend left] (v2);
\draw  (v3) edge (v2);
\draw  (v2) edge[bend left] (v1);
\draw  (v5) edge[bend left] (v6);
\draw  (v5) edge[bend right] (v6);
\draw  (v6) edge (v5);
\end{tikzpicture}
\]

The differential $\delta_1$ cannot change the core. I.~e., if it acts on a graph $\Gamma$, the result is a sum of graphs with the same core. 
It follows that $(\mathit{gr}\XGC, \delta_1)$ splits into subcomplexes of graphs having the same cores. Fix one such core. If the core has no input vertex, there are only two possibilities (by connectedness): (i) the core is the ``empty graph'' consisting of only the vertex $\vout$, or (ii) the core  
consists of the vertex $\vout$ and one tadpole at $\vout$. For case (i) there is (potentially) only one graph with that core, namely the empty graph itself. But this graph is excluded from $\XGC$ by definition. For case (ii) the subcomplex corresponding to that core is four dimensional. The four graphs generating it are depicted below. It is clear that the resulting complex is acyclic.

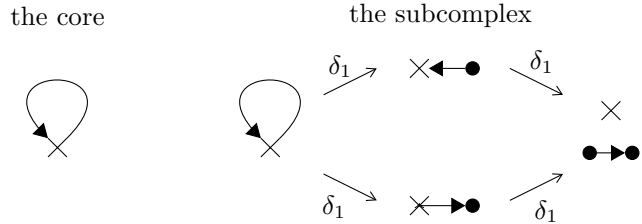
\[
\begin{tikzpicture}[every edge/.style={draw, -triangle 60}, scale=.7]

\node[xst] (v2) at (-5,0.5) {};
\node[xst] (v3) at (-1,0.5) {};
\node[xst] (v5) at (1.8,2) {};
\node[xst] (v10) at (1.8,-0.6) {};
\node[xst] (v11) at (5.4,1.2) {};
\node [int] (v9) at (2.8,-0.6) {};
\node [int] (v7) at (2.8,2) {};
\node at (-5,3) {the core};
\node at (2.2,3) {the subcomplex};
\draw [-angle 60](0,1.5) -- (1,2);
\draw [-angle 60](3.5,2) -- (4.5,1.5);
\draw [-angle 60](0,0) -- (1,-0.5);
\draw [-angle 60] (3.5,-0.5) -- (4.5,0);
\node at (4.1875,-0.7067) {$\delta_1$};
\node at (0.3138,2.0396) {$\delta_1$};
\node at (0.1981,-0.62) {$\delta_1$};
\node at (4.1008,2.1263) {$\delta_1$};

\draw [->] (v2) edge[loop, looseness=15] (v2);
\draw [->] (v3) edge[loop, looseness=15] (v3);
\draw [->] (v7) edge (v5);
\draw [->] (v10.base) edge (v9);
\node [int] (v1) at (5,0.4) {};
\node [int] (v4) at (5.8,0.4) {};
\draw  (v1) edge (v4);
\end{tikzpicture}
\]

Next consider a core with at least one input vertex.
Let us consider the subcomplex corresponding to this core. Focus on one vertex of the core, that is connected by $k$ incoming edges and $l$ outgoing edges to $\vout$. If $k,l>0$ graphs with this fixed core locally have one of the following four shapes:
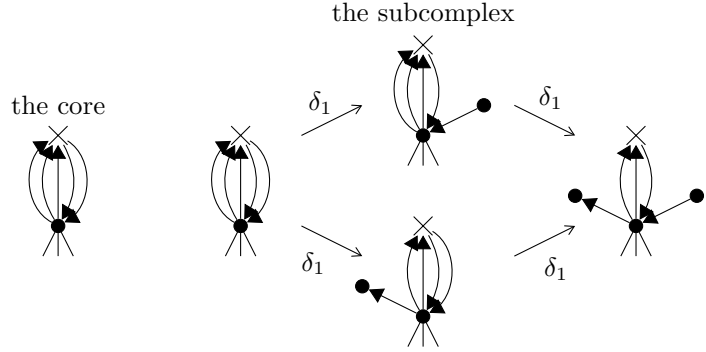
\[
\begin{tikzpicture}[every edge/.style={draw, -triangle 60}, scale=.8]

\node[xst] (v2) at (-4,1.5) {};
\node[xst] (v3) at (-1,1.5) {};
\node[xst] (v5) at (2,3) {};
\node[xst] (v10) at (2,0) {};
\node[xst] (v11) at (5.5,1.5) {};
\node [int] (v1) at (-4,0) {};
\node [int] (v4) at (-1,0) {};
\node [int] (v6) at (2,1.5) {};
\node [int] (v8) at (2,-1.5) {};
\node [int] (v12) at (5.5,0) {};
\node [int] (v14) at (4.5,0.5) {};
\node [int] (v13) at (6.5,0.5) {};
\node [int] (v9) at (1,-1) {};
\node [int] (v7) at (3,2) {};

\draw  (v1) edge[bend left] (v2);
\draw  (v1) edge[bend left=60] (v2);
\draw  (v1) edge (v2);
\draw  (v2) edge[bend left] (v1);
\draw  (v2) edge[bend left=60] (v1);

\draw  (v4) edge[bend left] (v3);
\draw  (v4) edge[bend left=60] (v3);
\draw  (v4) edge (v3);
\draw  (v3) edge[bend left] (v4);
\draw  (v3) edge[bend left=60] (v4);

\draw  (v6) edge[bend left] (v5);
\draw  (v6) edge[bend left=60] (v5);
\draw  (v6) edge (v5);
\draw  (v5) edge[bend left] (v6);

\draw  (v8) edge[bend left] (v10);
\draw  (v8) edge (v10);
\draw  (v10) edge[bend left] (v8);
\draw  (v10) edge[bend left=60] (v8);

\draw  (v12) edge[bend left] (v11);
\draw  (v12) edge (v11);
\draw  (v11) edge[bend left] (v12);

\draw  (v7) edge (v6);
\draw  (v8) edge (v9);

\draw  (v13) edge (v12);
\draw  (v12) edge (v14);
\node at (-4,2) {the core};
\node at (2,3.5) {the subcomplex};
\draw [-angle 60](0,1.5) -- (1,2);
\draw [-angle 60](3.5,2) -- (4.5,1.5);
\draw [-angle 60](0,0) -- (1,-0.5);
\draw [-angle 60] (3.5,-0.5) -- (4.5,0);
\node at (4.1875,-0.7067) {$\delta_1$};
\node at (0.3138,2.0396) {$\delta_1$};
\node at (0.1981,-0.62) {$\delta_1$};
\node at (4.1008,2.1263) {$\delta_1$};

\foreach \y in {.25,0,-.25}
{
\draw (v4.base) -- +(\y, -.5);
\draw (v8.base) -- +(\y, -.5);
\draw (v6.base) -- +(\y, -.5);
\draw (v12.base) -- +(\y, -.5);
\draw (v1.base) -- +(\y, -.5);
}
\end{tikzpicture}
\]
In case $k=0$ the top right two graphs are missing, if $l=0$ the bottom right two graphs are missing, and if $k=l=0$ the three rightmost graphs are absent from this picture. 


Now the subcomplex of xgraphs with a fixed core has the form
\[
 ( \otimes_v C_v )^G
\]
where $G$ is the symmetry group of the core. The tensor product runs over input vertices in (one representative of) the core and the $C_v$ are relatively simple complexes associated to the vertices. Concretely $C_v$ is of dimension 1,2 or 4, corresponding to the up to 4 graphs occurring in the above drawing. It is clear that $C_v$ is acyclic if the vertex $v$ in the core has special edges (i.e., edges to $\vout$) attached to it.
In case $v$ has no such edges, $c_v$ is one dimensional, hence the cohomology is also one dimensional.
Hence only graphs without special edges can contribute nontrivial cohomology classes of $(\mathit{gr}\XGC, \delta_1)$.
This shows Claim 2. 

Finally let us deduce Claim 1 from Claim 2. From Claim 2 and the completeness of the filtration it follows that the projection $\XGC\to \mathit{gr}^0\XGC / \delta_1 \mathit{gr}^1\XGC$ is a quasi-isomorphism.
By what is said above, the space $\delta_1 \mathit{gr}^1\XGC$ is composed of all graphs in $\mathit{gr}^0\XGC$ that have valence 1 vertices.
Hence the composition of maps of complexes
\[
\dGC \oplus \R S \to \XGC \to \mathit{gr}^0\XGC / \delta_1 \mathit{gr}^1\XGC
\]
is an isomorphism. This shows Claim 1 and the Proposition.
\hfill \qed

\section{The proof of Theorem \ref{thm:main}}
\label{sec:proof}
The proof of Theorem \ref{thm:main} is more or less a copy of M. Kontsevich's proof of Theorem \ref{thm:konts}, plus one additional combinatiorial step. Note that the general linear Lie (or $\Lie\{1\}$) algebras act on the Chevalley complex $C(\Tpoly, \hatTpoly)$ of $\Tpoly$, and on the sub-complexes $C(\Tpoly, \tildeTpoly)$, $C(\Tpoly, \Tpoly)$. Let us focus on the first of the three spaces to begin with. It is a projective limit of spaces
\[
V_n := C(\Tpoly^{(n)}, \hatTpoly^{(n)})
\]
on which $GL(n)$ acts in a reductive manner. Furthermore the action of $\gl_n$ is trivial on cohomology since it is given by a Cartan type formula ($L_x = d\circ \iota_x + \iota_x \circ d$, where $d$ is the Chevalley differential, $x\in \mathfrak{gl}_\infty$, $L_x$ is the action of $x$ and $\iota_x$ is the ``insertion operation'').
It follows that the cohomology of $V_n$ is the same as that of its invariant part $H(V_n)= H(V_n^{GL(n)})=H(V_n)^{GL(n)}$.
By classical invariant theory the elements in $V_n^{GL(n)}$ can be identified with certain graphs, which can be checked to be xgraphs, see also section \ref{sec:altdef}.
Hence $V_n$ is quasi-isomorphic to a certain quotient of $\fXGC$. It is a quotient, not the full space $\fXGC$, since due to finite dimensions some graphs act as zero. Let us call this quotient $\fXGC^{(n)}$.
We know that:
\begin{enumerate}
\item The maps $V_{n+1}\to V_n$ are surjective.
\item The maps $\fXGC^{(n)}\to V_n$ are quasi-isomorphisms.
\item The maps $\fXGC^{(n+1)}\to \fXGC^{(n)}$ are surjective.
\end{enumerate}
By the second assertion the mapping cones $\MC(\fXGC^{(n)}\to V_n)$ are acyclic for all $n$. By assertions 1 and 3 the mapping cones form a tower, with all morphisms surjective:
\[
\cdots \leftarrow \MC(\fXGC^{(n)}\to V_n) \leftarrow \MC(\fXGC^{(n+1)}\to V_{n+1})\leftarrow \cdots.
\]
Hence by \cite[Theorem 3.5.8]{weibel}, we can conclude that 
\begin{multline*}
0 = \lim_\leftarrow H(\MC(\fXGC^{(n)}\to V_n))
= H(\lim_\leftarrow \MC(\fXGC^{(n)}\to V_n))
\\
= H(\MC(\lim_\leftarrow \fXGC^{(n)}\to \lim_\leftarrow V_n)).
\end{multline*}
In other words the map 
\[
\lim_\leftarrow \fXGC^{(n)}\to C(\Tpoly, \hatTpoly)
\]
is a quasi-isomorphism.
But $\lim_\leftarrow  \fXGC^{(n)}$ is easily checked to be $\fXGC$. (Any xgraph acts non-trivially on some $\Tpoly^{(n)}$ for high enough $n$. The $n$ depends on the xgraph.)
Hence one of the three equalities of Theorem \ref{thm:main} then follows from Lemma \ref{lem:reduction} and Propositions \ref{prop:XGC} and \ref{prop:dGC}. 

For the sub-complex $C(\Tpoly, \tildeTpoly)$ the same argument goes through, one just has to replace $\hatTpoly^{(n)}$ by $\tildeTpoly^{(n)}$.
\hfill \qed

\begin{rem}
There is an alternative approach to the proof of Theorem \ref{thm:main} suggested by the anonymous referee of this article. First, one may show that the completed module of polyvector fields $\hatTpoly(\R^n)$ may be identified with the coinduced representation $\mathit{Coind}_{U( \Tpoly^{\geq 2}(\R^n))}^{U(\Tpoly^{\geq 1}(\R^n))}\R$, where $\Tpoly^{\geq 1}(\R^n)=\Tpoly(\R^n)/\R$. Then one can apply Shapiro's Lemma to see that $H(\Tpoly^{\geq 1}(\R^n), \hatTpoly(\R^n) )\cong H(\Tpoly^{\geq 2}(\R^n),\R )$. 
One may use M. Kontsevich's Theorem \ref{thm:konts} to compute $H(\Tpoly^{\geq 2}(\R^n),\R )$ in the stable limit.
Finally one may relate $H(\Tpoly^{\geq 1}(\R^n), \hatTpoly(\R^n) )$ and $H(\Tpoly(\R^n), \hatTpoly(\R^n) )$ with some extra work, leading to the appearance of the additional classes $S$ in Theorem \ref{thm:main}.
\end{rem}

\section{The cohomology of \texorpdfstring{$\Tpoly$}{Tpoly}.}
\label{sec:TpolyTpolycohom}
Let us consider the complex $C(\Tpoly, \Tpoly)$. Unfortunately, its cohomology is more subtle than that occurring in Theorem \ref{thm:main}. For example, consider a sequence of polynomials $p_1, p_2,\dots \in \Tpoly$ with $p_j\in \R[x_j,\dots, \xi_j, \dots]$ for $j=1,2,\dots$. Then the assignment
\[
\gamma \mapsto \sum_j \co{p_j}{\gamma}
\]
is well defined (only finitely many terms in the sum are nonzero for fixed $\gamma$) and defines a 1-cocyle in $C^1(\Tpoly,\Tpoly)$. In general it represents a non-trivial cohomology class if the sequence $\{p_n\}_n$ is infinite. Another sequence $\{p_n'\}_n$ of the above form gives rise to the same cohomology class if $\{p_n\}_n$ and $\{p_n'\}_n$ eventually agree, i.~e., if there is an $N$ such that $p_n=p_n'$ for $n>N$. In particular, the cohomology classes of this form are invariant under the action of $GL(\infty)$, though there is (in general) no $GL(\infty)$ invariant representative. Note however that two sequences can give rise to the same cohomology class even if they do not eventually agree. For example, one may shift sequences or add and subtract the same term from some elements of the sequence.
\begin{rem}
Note also that the product of two cohomolgy classes of the above form is trivial. Concretely, for sequences $\{p_n\}_n$ and $\{q_n\}_n$ as above (of homogeneous degrees) the 2-cocyle 
\[
(\gamma,\nu)\mapsto 
\sum_{i,j} (-1)^{(|q_j|+1)|\gamma|}\co{p_i}{\gamma}\co{q_j}{\nu}
+\sum_{i,j} (-1)^{(|q_j|+1+|\gamma|)|\nu|} \co{p_i}{\nu}\co{q_j}{\gamma}
\]
is the coboundary (up to global sign) of the cochain
\[
\gamma \mapsto
\sum_{i\leq j} p_i\co{q_j}{\gamma}
- (-1)^{|q_j||p_i|+|p_i|+|q_j|}
\sum_{i>j} q_j\co{p_i}{\gamma}\, .
\]
\end{rem}

\begin{ex}
The operation 
\[
T \colon \gamma \mapsto \sum_j \co{x^j \xi_j}{\gamma}
\]
obtained by taking the bracket with the Euler vector field represents a cohomology class in $H^1(\Tpoly, \Tpoly)$. Note that the Euler vector field does not belong to $\Tpoly$.
\end{ex}

I claim that the cohomology of $C(\Tpoly, \Tpoly)$ is generated (under cup products and up to completion) by the graph cohomology $H(\GC)$, the wheels $W_j$, $j=1,2,\dots$, the element $S$ as in Theorem \ref{thm:main} and the classes coming from sequences of polynomials as above.

However, let us consider only the invariant part $C(\Tpoly, \Tpoly)^{GL(\infty)}$ here, which is probably of higher practical interest anyways.

\begin{prop}
\label{prop:GLinftytpoly}
The cohomology of $C(\Tpoly, \Tpoly)^{GL(\infty)}$ can be identified with the space
\[
\mathbf{S}( \R S[-2] \oplus \R T[-2] \oplus \bigoplus_{n\equiv1\,\mathrm{ mod }\,4} \R W_n[-2]  \oplus H(\GC )[-2])[2]\, .
\]
Here $T$ is as in the example above and the remainder of the notation is the same as in Theorem \ref{thm:main}.
\end{prop}

\begin{rem}\label{rem:deg0}
 Note that in particular the degree 0 cohomology is spanned by $H(\GC )\cong \grt_1$ and by the classes $S$ and $T$.
\end{rem}

\begin{proof}[Proof of Proposition \ref{prop:GLinftytpoly}.]
Note that the cohomology in the Proposition is the cohomology of the subcomplex $\fXGC^{nt}\subset \fXGC$ spanned by graphs without special tadpoles. Hence it suffices to show that 
\[
C(\Tpoly, \Tpoly)^{GL(\infty)} \cong  \fXGC^{nt}\, .
\]
This follows essentially from classical invariant theory for $GL(\infty)$.
More precisely, let us write 
\[
C(\Tpoly, \Tpoly) = \lim_\leftarrow C(\Tpoly^{(n)}, \Tpoly).
\]
Then 
\[
C(\Tpoly, \Tpoly)^{GL(\infty)} = \lim_\leftarrow C(\Tpoly^{(n)}, \Tpoly)^{GL(n)}.
\]
Note that $C(\Tpoly^{(n)},\Tpoly^{(n)})$ carries an additional bigrading (as vector space) by the input and output degrees of maps. More precisely, the space of cochains with $r$ inputs can be written as 
\[
C^r(\Tpoly^{(n)},\Tpoly^{(n)}) = \prod_{d_i} \bigoplus_{d_o} C^r_{d_i,d_o}(\Tpoly^{(n)},\Tpoly^{(n)})
\]
where $C^r_{d_i,d_o}(\Tpoly^{(n)},\Tpoly^{(n)})$ is the subspace of chains that map homogeneous polynomials $\gamma_1,\dots, \gamma_r$ of total degree (number of $x$'s and $\xi$'s in all $\gamma$'s) $d_i$ to a polynomial of total degree $d_o$.
Each such subspace is finite dimensional since there are only finitely many possible monomials. There is a similar additional bigrading on $\fXGC$ and the map $\fXGC\to \Tpoly^{(n)}$ preserves the grading. In particular it follows that $\fXGC^{(n)}= \fXGC/I_n$, with $I_n$ the kernel of the previous map, inherits this bigrading. Concretely, a graph is of bidegree $(d_i, d_o)$ if the sum of valences of all input vertices is $d_i$ and the valence of the output vertex is $d_o$. We will call $\fXGC^{(n)}_{d_i, d_o}$ the space spanned by such graphs. 

Note furthermore that 
\[
C(\Tpoly^{(n)},\Tpoly) =  C(\Tpoly^{(n)},\Tpoly^{(n)})\hat \otimes \R[x_{n+1},\dots, \xi_{n+1}, \dots]
\]
for a suitably completed tensor product $\hat \otimes$ and hence
\begin{align*}
C(\Tpoly^{(n)},\Tpoly)^{GL(n)} 
&=  C(\Tpoly^{(n)},\Tpoly^{(n)})^{GL(n)}\hat \otimes \R[x_{n+1},\dots, \xi_{n+1}, \dots] \\
&= \fXGC^{(n)} \hat \otimes \R[x_{n+1},\dots, \xi_{n+1}, \dots].
\end{align*}
This reduces our goal to showing that 
\[
\lim_\leftarrow 
\fXGC^{(n)} \hat \otimes \R[x_{n+1},\dots, \xi_{n+1}, \dots]
= \fXGC^{nt}.
\]
Consider just the grading by input degrees. It allows us to reduce the statement further to showing that 
\[
\lim_\leftarrow 
\fXGC^{(n)}_{d_i} \hat \otimes \R[x_{n+1},\dots, \xi_{n+1}, \dots]
= \fXGC^{nt}_{d_i}
\]
where the subscripts indicate that we restrict to subspaces spanned by graphs with fixed input degree (the sum of valences of all input vertices) $d_i$. Note that for $d_i$ fixed and $n$ big enough $\fXGC^{(n)}_{d_i}= \fXGC_{d_i}$. It suffices to consider terms for $n$ big enough.
Consider a graph $\Gamma$ with special tadpole in $\fXGC_{d_i}$ and let $\Gamma'$ be the same graph with special tadpole deleted. For polynomials $\gamma_n, \nu_n\in \R[x_{n+1},\dots, \xi_{n+1}, \dots]$ the map 
\[
\fXGC^{(n)}_{d_i} \hat \otimes \R[x_{n+1},\dots, \xi_{n+1}, \dots]
\to 
\fXGC^{(n-1)}_{d_i} \hat \otimes \R[x_{n},\dots, \xi_{n}, \dots]
\]
maps
\[
\Gamma\otimes \gamma_n + \Gamma'\otimes \nu_n
\mapsto
\Gamma\otimes \gamma_n +
\Gamma'\otimes ((x^n\xi_n)\gamma_n +\nu_n)\, .
\]
Hence the coefficients $\gamma_n$, $\nu_n$ of graphs $\Gamma$, $\Gamma'$ for an element of the inverse limit must satisfy the equations (for $n\gg 0$)
\begin{align*}
\gamma_{n-1} &= \gamma_n\\
\nu_{n-1} &= (x^n\xi_n)\gamma_n +\nu_n\, .
\end{align*}
These equations are impossible to satisfy unless $\gamma_n=0$, since the element $\sum_{k\geq n} x^k\xi_k$ is not contained in $\R[x_{n},\dots, \xi_{n}, \dots]$.
In this case the collection $(\nu_n)_n$ has to define an element of $\lim_\leftarrow \R[x_{n},\dots, \xi_{n}, \dots]$, i.~e., the $\nu_n$ must be a constant. The claim of the Proposition follows.
\end{proof}

\section{Relation to stable formality morphisms}
\label{sec:discussion}
%
%
%

In this section we relate our results to the recent work by V. Dolgushev \cite{vasilystable} and indicate an alternative, more intrinsic proof of the main theorem of loc. cit. This alternative proof is already hinted at in some hidden form in \cite[section 4.2]{K2}.

\subsection{Automorphisms of \texorpdfstring{$\Tpoly$}{Tpoly}, and stable formality morphisms}
Let $\Dpoly^{(n)}$ be the space of polynomial polydifferential operators on $\R^n$. By the Hochschild-Kostant-Rosenberg (HKR) Theorem it is quasi-isomorphic to $\Tpoly^{(n)}$ as a complex. Since taking direct limits commutes with taking cohomology, we see that the stable version
\[
\Dpoly = \lim_\rightarrow \Dpoly^{(n)}
\]
is quasi-isomorphic to $\Tpoly$ as a complex.
By M. Kontsevich's Formality Theorem the two spaces are also quasi-isomorphic as $(\Lie\{1\})_\infty$ algebras.
One can consider the space $M$ of $(\Lie\{1\})_\infty$ quasi-isomorphisms 
$\Tpoly\to \Dpoly$ whose unary component is the Hochschild-Kostant-Rosenberg (HKR) morphism.

Let $\Aut_\infty(\Tpoly)$ be the group of $(\Lie\{1\})_\infty$ automorphisms of $\Tpoly$. Since $\Tpoly$ has zero differential, the leading component of any $(\Lie\{1\})_\infty$ automorphism is an honest $\Lie\{1\}$ algebra automorphism. Denote the kernel of the morphism  $\Aut_\infty(\Tpoly)\to \Aut(\Tpoly)$ by $G$. There is a natural right action of $G$ on $M$.

We may require one of the following additional properties from our stable formality morphisms or automorphisms:
\begin{itemize}
\item There are natural actions of $GL(\infty)$ on $\Tpoly$ and $\Dpoly$. We may ask the formality morphisms and automorphisms of $\Tpoly$ to be equivariant with respect to this action. We obtain the space $M^{GL}\subset M$ of equivariant formality morphisms, with a right action of the group $G^{GL}$ of equivariant $(\Lie\{1\})_\infty$ automorphisms of $\Tpoly$.
\item Both $\Tpoly$ and $\Dpoly$ are left $\mO = \R[x^1, x^2, \dots]$ modules. We may ask the components of the formality morphisms and automorphisms to be (poly)differential operators between $\mO$ modules. 

\item We may combine the above and require that the components of our $(\Lie\{1\})_\infty$ automorphisms and formality morphisms are $GL(\infty)$ invariant constant coefficient (i.~e., translation invariant) polydifferential operators. We call the  
set of such formality morphisms $M^{pg}\subset M^{GL}\subset M$ and the group $G^{pg}\subset G^{GL}\subset G$. For example, M. Kontsevich's formality morphism is an element of $M^{pg}$.
\end{itemize}

\subsection{Homotopy}
We will use the following ``cheap'' notion of homotopy between  $(\Lie\{1\})_\infty$ morphisms.

\begin{defi}
Let $\alg g$ and $\alg h$ be $(\Lie\{1\})_\infty$ algebras and let $f, f'\colon \alg g\to \alg h$ be $(\Lie\{1\})_\infty$ morphisms. We say that $f$ and $f'$ are \emph{homotopic on the nose} if there is an 
$(\Lie\{1\})_\infty$ morphism
\[
\alpha : \alg g \to \alg h[t, dt]
\]
such that the restriction $\alpha|_{t=0} = f$ and $\alpha|_{t=1}=f'$.

We say that $f, f'$ are homotopic if there are morphisms $f_1, f_2,\dots, f_n:G\to H$ for some $n$ such that $f$ is homotopic on the nose to $f_1$, $f_1$ is homotopic on the nose to $f_2$ etc., and $f_n$ is homotopic on the nose to $f'$. In other words, we define the relation of homotopy as the transitve closure of the relation of homotopy on the nose.
\end{defi}

This defines an equivalence relation on the space of morphism $\alg g\to \alg h$.\footnote{In fact, being homotopic on the nose is already an equivalence relation, identical to being homotopic as defined above, see, e.g., \cite[Theorem 2.1.1]{fukayacomo}. }
Furthermore, let $\alg k$ be another $(\Lie\{1\})_\infty$ algebra.
Suppose there are homotopic morphisms $f,f': \alg g\to \alg h$ and $g, g':\alg h\to \alg k$. Then clearly $g\circ f$ and $g'\circ f'$ are homotopic.

If in the previous definition $\alg g$ and $\alg h$ are copies of $\Tpoly$ or $\Dpoly$  and we consider morphism $f, f'$ of one of the restricted forms of the previous subsection, we will require that the homotopies also have the restricted form.

\begin{defi}
Let $\alg g,\alg h$ and $f,f'$ be as in the previous definition.
\begin{itemize}
\item Suppose the maps of complexes $\alg g\to \alg h$ induced by $f,f'$ (i.~e., their unary components) agree. We say that $f, f'$ are homotopic on the nose with fixed unary component if the homotopy $\alpha$ may be chosen such that, in addition to the requirements of the previous definition, the restrictions at fixed $t$, i.e., $\alpha|_t:\alg g\to \alg h$, have the same unary component as $f$ and $f'$.
\item Let $\alg g, \alg h$ again be copies of $\Tpoly$ or $\Dpoly$, and assume that $f,f'$ are $GL(\infty)$ equivariant morphisms. Then we say that  $f,f'$ are homotopic on the nose as $GL(\infty)$ equivariant morphisms if $\alpha$ as above may be chosen to be $GL(\infty)$ equivariant.
\item Restrict to $\alg g, \alg h$ being copies of $\Tpoly$ or $\Dpoly$, and assume that $f,f'$ are polydifferential morphisms. Then we say that  $f,f'$ are homotopic on the nose as polydifferential morphisms if $\alpha$ as above may be chosen to be a polydifferential morphism.
\item etc...
\end{itemize}
In each case we define the relation of homotopy between morphisms as the transitive closure of the relation of homotopy on the nose.
\end{defi}

\begin{defi}
Let $H\subset G=\ker (\Aut_\infty(\Tpoly)\to \Aut(\Tpoly))$ be the normal subgroup of automorphisms which are homotopic (with fixed unary part) to the identity morphism. The group of $(\Lie\{1\})_\infty$ automorphisms of $\Tpoly$ up to homotopy is defined to be $\bar G:=G/H$. Similarly we define $\bar G^{GL(\infty)}$ and $\bar G^{pg}$, as quotients of the subgroups of $GL(\infty)$ or constant coefficient polydifferential and $GL(\infty)$ equivariant morphisms, modulo the appropriate homotopy relation. 

We define $\bar M$ to be the quotient of $M$ obtained by identifying homotopic formality morphisms. Similarly, we define $\bar M^{GL(\infty)}$ and $\bar M^{pg}$.
\end{defi}

The right actions of $G$ on $M$ descends to the quotients, so that we obtain an action of $\bar G$ on $\bar M$. Similarly $\bar G^{GL(\infty)}$ and $\bar G^{pg}$ act on $\bar M^{GL(\infty)}$ and $\bar M^{pg}$ (from the right).

\subsection{V. Dolgushev's Theorem}
Formality morphisms in $M^{pg}$ are called \emph{stable formality morphisms} in \cite{vasilystable}. 
Note that the (pro-nilpotent) Lie algebra of degree zero cocycles of $\fGC$ acts on $\Tpoly$ by $(\Lie\{1\})_\infty$-derivations, via the map
$\fGC\to C(\Tpoly,\Tpoly)$ discussed above. Hence it also acts on the space of formality morphisms $M$.
Since the corresponding $(\Lie\{1\})_\infty$-derivations are all given by constant coefficient polydifferential operators this action descends to an action on $M^{pg}$.
Since exact cocyles act in a null-homotopic way, the action induces an action of the  Lie algebra $H^0(\fGC)$ on $\bar M^{pg}$. By integration we also 
obtain an action of the exponential group $\exp(H^0(\fGC))$ on $\bar M^{pg}$.
The main Theorem of \cite{vasilystable} is the following:
\begin{thm}[V. Dolgushev \cite{vasilystable}]\label{thm:vasstable}
 $\bar M^{pg}$ is an $\exp(H^0(\fGC))$-torsor.
\end{thm}

Here we want to remark that this result can essentially be obtained also as an application of the following general principle:
Suppose we have two isomorphic objects $o$, $o'$ in some category, then for some group $G$ it is equivalent to say that (i) $\Aut(o)\cong G$ or that (ii) the space of isomorphisms $o\to o'$ is a $G$-torsor.

In the present setting we want to take for the category the homotopy category of Lie algebras, $o=\Tpoly$ and $o'=\Dpoly$.
There is however a further small technical difficulty in that we defined the set of stable formality morphisms only as a subset of the set of all quasi-isomorphisms.
This technical difficulty may however be settled, as we will indicate in the following sections.

\subsection{Inversion}
Consider the Hochschild-Kostant-Rosenberg quasi-isomorphism (of complexes)
\[
\phi \colon \Tpoly \to \Dpoly \, .
\]
It has an explicit left inverse $\phi^{-1}$, i.e., $\phi^{-1}\circ\phi=\id_{\Tpoly}$. Furthermore, there is a homotopy $h$ given by explicit recursion relations such that $\id - \phi\circ\phi^{-1}=hd_H+d_H h$ where $d_H$ is the Hochschild differential on $\Dpoly$. In addition $h$ has the following properties:
\begin{enumerate}
\item $\phi^{-1}$ and $h$ commute with the $GL(\infty)$ action.
\item $\phi^{-1}$ and $h$ are $\mO$-linear, i.e., degree zero differential operators.
\item $\phi^{-1}$ and $h$ are a constant coefficient (degree zero) differential operators. 
\end{enumerate}
Explicit recursion formulas for $\phi^{-1}$ and for the homotopy $h$ can be extracted from \cite{vergne}. 

Now let $F\in M$ be a (stable) formality morphism $F\colon \Tpoly\to\Dpoly$. 
Using $\phi^{-1}$ and $h$ from above one can construct (by explicit recursion formulas) a left inverse $F^{-1}$ to $F$, i.e., $F^{-1}\circ F=\id_{\Tpoly}$. The components of $F^{-1}$ are all given by composing components of $F$ with copies of $\phi$, $\phi^{-1}$, $h$ and the Lie brackets. In particular, if $F$ falls in any of the restricted classes from above ($GL(\infty)$ equivariant, polydifferential, etc.), then $F^{-1}$ will fall in the same class because of the properties of $\phi^{-1}$ and $h$ asserted above.
Furthermore, the composition $F\circ F^{-1}$ is homotopic on the nose to $\id_{\Dpoly}$, and the homotopy is defined by suitable compositions of 
$\phi$, $\phi^{-1}$, $h$ and the Lie brackets. 
In particular, if $F$ falls in any of the restricted classes from above ($GL(\infty)$ equivariant, polydifferential, etc.) then $F\circ F^{-1}$ is homotopic on the nose to $\id_{\Dpoly}$ in that restricted class.

\subsection{Formality morphisms up to homotopy}

\begin{lemma}
\label{lem:torsor}
 $\bar M$ is a $\bar G$-torsor, $\bar M^{GL(\infty)}$ is a $\bar G^{GL(\infty)}$-torsor and $\bar M^{pg}$ is a $\bar G^{pg}$-torsor.
\end{lemma}
\begin{proof}
It is a purely formal statement (and standard fact), but let us do the argument for $\bar M$ and $\bar G$. 

Faithfulness: Let $f\in G$, $F\in M$ and suppose $Ff$ is homotopic to $F$. Then we need to show that $f$ is homotopic to the identity. Applying $F^{-1}$ we see that $F^{-1}Ff$ is homotopic to $F^{-1}F$. But $F^{-1}F$ is homotopic to the identity and hence $f$ is, too, being homotopic to $F^{-1}Ff$.

Transitivity: Let $F, F'\in M$. Then we need to show that there is an $f$ such that $Ff$ is homotopic to $F'$. Take $f=F^{-1} F'$, so we need to show that $F F^{-1} F'$ is homotopic to $F'$. It is sufficient to show that $F^{-1} F F^{-1} F'$ is homotopic to $F^{-1} F'$. But this is clear since both morphisms are equal.
\end{proof}

\subsection{Reduction to \texorpdfstring{$\Lie\{1\}$}{Lie\{1\}} algebra cohomology}
$G$ is a pro-unipotent group and thus the exponential group of its Lie algebra. The Lie algebra is the space of $(\Lie\{1\})_\infty$ derivations of $\Tpoly$, with zero unary component. This space in turn can be seen as the closed degree zero elements in $C(\Tpoly, \Tpoly)$, with vanishing unary (and zero-ary) part.
An element $\exp(x)\in G$, with $x$ a degree zero cocyle in $C(\Tpoly, \Tpoly)$, is homotopic to the identity iff $x$ is exact, cf. Lemma \ref{lem:linftyhomotopy} in Appendix \ref{app:linftyhomotopy}. 
Hence the Lie algebra of $H\subset G$ (with $H$ as above) is the Lie algebra of degree zero cocycles in $C(\Tpoly, \Tpoly)$, with vanishing unary and zero-ary part.
The Lie algebra of $\bar G$ is hence given by the subspace of $H^0(\Tpoly, \Tpoly)$ of elements without unary or zero-ary part. Let us denote this subspace by $H^0(\Tpoly, \Tpoly)'$.
Hence, using Lemma \ref{lem:torsor} we see that $\bar M$ is a torsor over the exponential group of $H^0(\Tpoly, \Tpoly)'$.
Similar reasoning applies for the $GL(\infty)$ equivariant or constant coefficient polydifferential cases, so we obtain the following statements:

\begin{itemize}
\item The space $\bar M^{GL(\infty)}$ is a torsor over 
$H^0( C(\Tpoly, \Tpoly)^{GL(\infty)} )'$, where the $'$ again means that we restrict to those derivations having vanishing unary and zero-ary part.
\item The space $\bar M^{cg}$ is a torsor over the exponential group of the Lie algebra
$H^0( C_{polydiff}(\Tpoly, \Tpoly)^{GL(\infty)} )'$. Here $C_{polydiff}(\Tpoly, \Tpoly)\subset C(\Tpoly, \Tpoly)$ are those chains given by $\mO$-polydifferential operators, and the $'$ again means that we require vanishing unary and zero-ary part.
\end{itemize}

\subsection{Reduction to graph cohomology}

As we saw above (or by definition), $C(\Tpoly, \Tpoly)^{GL(\infty)}\cong \fXGC^{nt}$. It is not hard to see that the graphs that give rise to polydifferential morphism are those in the image of $\fGC$. Hence, combining the statements above we obtain the following result:
\begin{itemize}
\item The space $\bar M^{GL(\infty)}$ is a torsor over the exponential group of
$H^0( \fXGC )'\cong H^0(\fGC)\cong H^0(\GC) \cong\grt_1$, i.~e., over the Grothendieck-Teichm\"uller group.
\item The space $\bar M^{pg}$ is a torsor over the exponential group of $H^0(\fGC)\cong\grt_1$, i.~e., over the Grothendieck-Teichm\"uller group.
\end{itemize}
In the first case the $'$ again means that we restrict to those elements having vanishing unary and zero-ary part. In other words, we forbid cohomology classes represented by graphs with less than 2 input vertices. In the second case this is not necessary since there are no such classes.
We furthermore used Remark \ref{rem:deg0}.

This recover the result of \cite{vasilystable} along a different (less combinatorial) route.

\appendix

\section{Homotopy theoretic Lemma}\label{app:linftyhomotopy}
In the proof of Theorem \ref{thm:vasstable} we sketched above the following homotopy theoretic lemma is used.

\begin{lemma}\label{lem:linftyhomotopy}
Let $\alg g$ be an $L_\infty$-algebra and $D_1$ and $D_2$ two $L_\infty$ derivations of $\alg g$, i.e., coderivations of the cofree cocommutative coalgebra $\mathbf{S}(\alg g[1])$. Suppose that these coderivations do not have linear terms, i.e., $D_1|_{\alg g} = D_2|_{\alg g} = 0$, so that the exponential maps ($L_\infty$ morphisms of $\alg g$) $\exp(D_1)$ and $\exp(D_2)$ are defined. Then $\exp(D_1)$ is homotopy equivalent to $\exp(D_2)$, with a homotopy $F:\alg g\to \alg g[t,dt]$ whose linear term is the identity at each $t$ if and only if $D_1$ is cohomologous to $D_2$.
\end{lemma}
\begin{proof}
First, suppose that $D_1$ and $D_2$ are cohomologous, say $D_2-D_1 = [d, x]$ for some (pre-)$L_\infty$ derivation $x$.
Then consider the derivation $D:= tD_2 + (1-t)D_1 + x dt$. Its exponential $\exp(D)$ is an $L_\infty$ automorphism of $\alg g[t, dt]$ that restricts to $\exp(D_1)$ and $\exp(D_2)$ at $t=0$ and $t=1$. Composing with the ``constant'' map $\alg g\to \alg g[t, dt]$ we obtain the desired homotopy $\alg g\to \alg g[t,dt]$ between $\exp(D_1)$ and $\exp(D_2)$.

Conversely, suppose that $F_0:=\exp(D_1)$ and $F_1:=\exp(D_2)$ are homotopic via some $L_\infty$ morphism $F : \alg g\to \alg g[t,dt]$ restricting to $F_0$ and $F_1$ at $t=0,1$ as in the Lemma.
Concretely, $F$ consists of a family $F_t:\alg g \to \alg g$ of $L_\infty$ morphisms together with a family $H_t$ of coderivations of $F_t$, such that 
\[
\dot{ F}_t=[d,H_t]. 
\]
It follows that $\dot{F}_t= F_t [d,h_t]$, with $h_t:=F_t^{-1}H_t$ a coderivation of $\alg g$.

Now consider the $L_\infty$ derivation $D_t:=\log(F_t)$. Taking the logarithm is possible since $F_t$ is assumed to have a fixed linear term equal to the identity map.
Clearly 
\[
 D_2-D_1 = \int_0^1 \dot{D}_t dt
\]
and hence it is sufficient to show that $\dot{D}_t$ is exact.
But 
\begin{align*}
 \dot{D}_t &= \frac{d}{dt}\log(F_t)
 = \sum_{n\geq 1} \frac 1 n \sum_{k=0}^{n-1} (1-F_t)^k F_t [d,h_t] (1-F_t)^{n-k-1}
 \\&= [d, \sum_{n\geq 1} \frac 1 n \sum_{k=0}^{n-1} (1-F_t)^k F_t h_t (1-F_t)^{n-k-1}].
\end{align*}
One checks that 
\[
  \tilde h_t := \sum_{n\geq 1} \frac 1 n \sum_{k=0}^{n-1} (1-F_t)^k F_t h_t (1-F_t)^{n-k-1}
\]
is a derivation and hence $\dot{D}_t=[d,\tilde h_t]$ is exact as required.
Indeed, to see that $\tilde h_t$ is a derivation it is sufficient to verify that the element
\[
 Z:=\sum_{n\geq 1} \frac 1 n \sum_{k=0}^{n-1} (1-e^X)^k e^X Y (1-e^X)^{n-k-1}
\]
in the completed universal enveloping algebra of a free Lie algebra in two symbols $X$ and $Y$ is primitive, i.e., its coproduct satisfies $\Delta Z=1\otimes Z+Z\otimes 1$.
But $Z=\frac{d}{dt}\mid_{t=0} \log(e^X e^{tY})$ is a derivative at $t=0$ of a family of primitive elements and hence primitive.
\end{proof}

\bibliographystyle{plain}

\end{document}